\documentclass[11pt]{article}

\usepackage{ulem}
\usepackage{framed}

\usepackage{amsthm}
\usepackage{amssymb}
\usepackage{amscd}
\usepackage{amsmath}
\usepackage[all]{xy}
\usepackage[top=30truemm,bottom=30truemm,left=25truemm,right=25truemm]{geometry}
\usepackage{comment}
\usepackage{algorithmic,algorithm}
\usepackage{here}
\usepackage{graphicx}
\usepackage{color} %
\usepackage{enumerate} %
\usepackage{listliketab}
\usepackage{multirow}


\makeatletter
  
  \@addtoreset{algorithm}{subsection}
\makeatother

\theoremstyle{definition}

\theoremstyle{definition}

\theoremstyle{plain}

\theoremstyle{plain}
\newtheorem{theor}{Theorem}

\theoremstyle{plain}
\newtheorem{coror}{Corollary}

\theoremstyle{plain}

\theoremstyle{plain}

\theoremstyle{plain}
\newtheorem{thm}{Theorem}[subsection]

\theoremstyle{definition}

\theoremstyle{definition}

\theoremstyle{definition}

\theoremstyle{definition}

\theoremstyle{definition}
\newtheorem{rem}[thm]{Remark}

\theoremstyle{plain}
\newtheorem{prop}[thm]{Proposition}

\theoremstyle{plain}
\newtheorem{lem}[thm]{Lemma}

\theoremstyle{plain}
\newtheorem{cor}[thm]{Corollary}

\theoremstyle{definition}

\theoremstyle{definition}

\theoremstyle{definition}

\theoremstyle{definition}

\theoremstyle{definition}

\numberwithin{equation}{subsection}

\def\qed{\hfill $\Box$}

\def\F{\mathbb{F}}

\newcommand{\GL}{\operatorname{GL}}

\newcommand{\Aut}{\operatorname{Aut}}

\makeatletter 
\long\def\@makefntext#1{\parindent 1em\noindent 
\@hangfrom{\hbox to 1.8em{\hss $^{\@thefnmark}$}}#1}
\makeatother
\makeatletter
\def\@seccntformat#1{\csname the#1\endcsname. }
\makeatother
\makeatletter
\renewcommand\section{\@startsection {section}{1}{\z@}%
 {-3.5ex \@plus -1ex \@minus -.2ex}%
 {2.3ex \@plus.2ex}%
 {\normalfont\large\bfseries}}
\makeatother

\newcommand{\keywords}[1]{\textbf{{Keywords: }} #1}
\newcommand{\MSC}[1]{\textbf{{2010 Mathematical Subject Classification: }} #1}
\begin{document}

\title
{\bf Algorithmic study of superspecial hyperelliptic curves over finite fields}
\author
{Momonari Kudo\thanks{Kobe City College of Technology} \thanks{Institute of Mathematics for Industry, Kyushu University.}
\ and Shushi Harashita\thanks{Graduate School of Environment and Information Sciences, Yokohama National University.}}

\maketitle


\begin{abstract}
This paper presents algorithmic approaches to study {\it superspecial} hyperelliptic curves.
The algorithms proposed in this paper are: an algorithm to enumerate {\it superspecial} hyperelliptic curves of genus $g$ over finite fields $\mathbb{F}_q$, and an algorithm to compute the automorphism group of a (not necessarily superspecial) hyperelliptic curve over finite fields.
The first algorithm works for any $(g,q)$ such that $q$ and $2g+2$ are coprime and $q>2g+1$.
As an application, we enumerate {\it superspecial} hyperelliptic curves of genus $g=4$ over $\mathbb{F}_{p}$ for $11 \leq p \leq 23$ and over $\mathbb{F}_{p^2}$ for $11 \leq p \leq 19$ with our implementation on a computer algebra system Magma.
Moreover, we found {\it maximal} hyperelliptic curves and {\it minimal} hyperelliptic curves over $\mathbb{F}_{p^2}$ from among enumerated superspecial ones.
The second algorithm computes an automorphism as a concrete element in (a quotient of) a linear group in the general linear group of degree $2$.

\\
\MSC 14G05, 14G15, 14G50, 14H45, 14Q05, 68W30
\end{abstract}

\section{Introduction}\label{sec:intro}
By a curve, we mean a projective, geometrically irreducible, and non-singular algebraic curve.
Let $C$ be a curve of genus $g$ over a field $K$ of positive characteristic $p>0$.
We call $C$ {\it superspecial} ({\it s.sp}.\ for short) if its Jacobian variety is $\overline{K}$-isomorphic to the product of $g$ supersingular elliptic curves, where $\overline{K}$ denotes the algebraic closure of $K$.

The problem which we mainly consider in this paper is to enumerate $K$-isomorphism classses of s.sp.\ curves of genus $g$ over the finite field $\mathbb{F}_{q}$ of $q$ elements, where $q$ is a power of $p$.
Note that it suffices to consider the case of $q = p$ and $p^2$ since the number of isomorphism classes of s.sp.\ curves over $\mathbb{F}_{p^a}$ depends on the parity of $a$ (cf. \cite[Proposition 2.3.1]{KH17-2}).
If $g \leq 3$, there are some theoretical approaches based on Torelli's theorem to find s.sp.\ curves (cf.\ \cite{Deuring}, \cite[Prop. 4.4]{XYY16} for $g=1$, \cite{HI}, \cite{IK}, \cite{Serre1983} for $g=2$, and \cite{Hashimoto}, \cite{Ibukiyama} for $g=3$).
Different from the case of $g \leq 3$, these approaches are considered to be not so effective for $g \geq 4$ by the following reason: The dimension of the moduli space of curves of genus $g \geq 4$ is strictly less than that of the moduli space of principally polarized abelian varieties of dimension $g$.

In the {\it non-hyperelliptic} case for $g \geq 4$, computational approaches to enumerate s.sp.\ curves were proposed, and the enumeration in some small particular characteristic has been completed (cf.\ \cite{KH16}, \cite{KH17-2}, \cite{KH17a} for $g=4$, and \cite{KH18tri} for $g=5$).
In particular, the isomorphism classes of s.sp.\ non-hyperelliptic curves of genus $4$ over $\mathbb{F}_q$ are determined for $q= 5^{ 2e -1 }$, $5^{2 e}$, $7^{2 e-1}$, $7^{2 e}$ and $11^{2e -1}$, where $e$ is a natural number.

A fascinating fact in the hyperelliptic case is that the existence of a s.sp.\ hyperelliptic curve of genus $g$ in characteristic $p$ implies that of
 a maximal (resp. minimal) curve of genus $g$ over $\F_{p^2}$,
see \cite[Subsection 2.2]{KH18} for a review of this fact.
Ekedahl \cite[Theorem 1.1]{Ekedahl} showed $p \geq 2 g + 1$ if a s.sp.\ hyperelliptic curve exists for $(g,p) \neq (1,2)$.
While the existence of s.sp.\ hyperelliptic curves of given genus is known for many $p$ with some congruent relations (e.g., \cite{Taf}, \cite{Taf2}), the enumeration of s.sp.\ ones of genus $g \geq 4$ has not been completed yet even for small particular $p$.

This paper is the full-version of our conference paper \cite{KH18} which enumerates s.sp.\ hyperelliptic curves of genus $4$ for $q = 11^{2e - 1}$, $11^{2 e}$, $13^{2e-1}$, $13^{2e}$, $17^{2 e - 1}$, $17^{2 e}$ and $19^{2 e -1}$.
The following (Theorems \ref{MainTheorem11} and \ref{MainTheorem22}) are the main theorems of \cite{KH18}:

\begin{theor}[\cite{KH18}, Theorem 1]\label{MainTheorem11}
There is no s.sp.\ hyperelliptic curve of genus $4$ in characteristic $11$ and $13$.
\end{theor}

\begin{theor}[\cite{KH18}, Theorem 2]\label{MainTheorem22}
There exist exactly five $($resp.\ $25)$ s.sp.\ hyperelliptic curves of genus $4$ over $\mathbb{F}_{17}$ $($resp.\ $\mathbb{F}_{17^2})$, up to isomorphism over $\mathbb{F}_{17}$ $($resp.\ $\mathbb{F}_{17^2})$.
Moreover, there exist exactly two s.sp.\ hyperelliptic curves of genus $4$ over the algebraic closure in characteristic $17$ up to isomorphism.
\end{theor}

In particular, Theorem \ref{MainTheorem11} relaxes the restriction on non-hyperelliptic curves in \cite[Theorem B]{KH17-2} (or \cite[Main Theorem]{KH17a}).

\begin{coror}[\cite{KH18}, Corollary 2]
There exist exactly $30$ $($resp.\ nine$)$ s.sp.\ curves of genus $4$ over $\mathbb{F}_{11}$, up to isomorphism over $\mathbb{F}_{11}$ $($resp.\ $\overline{\mathbb{F}_{11}})$.
\end{coror}

Additional and new results, that are not given in \cite{KH18}, of this paper are as follows: 
\begin{enumerate}
\item Complete proofs of computational results in \cite{KH18},
\item New results on enumeration for $q=19^2$ and $23$ (Theorems \ref{MainTheorem33} and \ref{MainTheorem44} below),
\item Computation of automorphism groups of enumerated s.sp.\ hyperelliptic curves.
\end{enumerate}

\begin{theor}\label{MainTheorem33}
There exist exactly $12$ $($resp.\ $25)$ superspecial hyperelliptic curves of genus $4$ over $\mathbb{F}_{19}$ $($resp.\ $\mathbb{F}_{19^2})$ up to isomorphism over $\mathbb{F}_{19}$ $($resp.\ $\mathbb{F}_{19^2})$.
Moreover, there exist exactly two superspecial hyperelliptic curves of genus $4$ over the algebraic closure in characteristic $19$ up to isomorphism.
\end{theor}

\begin{theor}\label{MainTheorem44}
There exist exactly $14$ superspecial hyperelliptic curves of genus $4$ over $\mathbb{F}_{23}$ up to isomorphism over $\mathbb{F}_{23}$.
Moreover, there exist exactly four superspecial hyperelliptic curves of genus $4$ over $\mathbb{F}_{23}$ up to isomorphism over the algebraic closure.
\end{theor}

The rest of this paper is organized as follows.
Section \ref{sec:pre} gives a review of general facts on hyperelliptic curves over finite fields,
In Section \ref{sec:enume}, we review the enumeration method given in \cite{KH18}.
The method consists of the following three ingredients: (A) Algorithm to list up s.sp.\ hyperelliptic curves, (B) Reduction of defining equations of hyperelliptic curves, and (C) Isomorphism testing.
Section \ref{sec:main} gives complete proofs of computational results in \cite{KH18}, and new results on enumeration for $q=19^2$ and $23$.
Section \ref{sec:aut} studies automorphism groups of enumerated s.sp.\ hyperelliptic curves.
Specifically, we give an algorithm to compute the automorphism group of a (not necessarily s.sp.) hyperelliptic curve.
Note that in this paper we do not mention the asymptotic complexity but the practicality of our algorithms only.

\section{Preliminaries}\label{sec:pre}

In this section, we review
a realization of hyperelliptic curves,
a criterion for their superspeciality and
a method to enumerate s.sp.\ hyperelliptic curves.
\subsection{Hyperelliptic curves}\label{subsec:hype}
Let $K$ be a field.
Let $C$ be a hyperelliptic curve over $K$, i.e.,
a curve over $K$ admitting a morphism over $K$ of degree $2$
from $C$ to the projective line $\mathbf{P}^1$.
As seen in \cite[Subsection 2.1]{KH18},
if the cardinality of $K$ is greater than $2g+1$,
then $C$ is realized as
the desingularization of the homogenization of
\begin{equation}
y^2 = f(x),
\end{equation}
where $f(x)$ is a polynomial over $K$
of degree $2g+2$
with non-zero discriminant.



The next lemma tells us when two hyperelliptic curves $C_1$ and $C_2$ are isomorphic.

\begin{lem}[cf. \cite{KH18}, Lemma 2.1]\label{lem:isom}
Let $f_1(x)$ and $f_2(x)$ be elements of $K[x]$ of degree $2g+2$.
Let $C_1$ and $C_2$ be the hyperelliptic curves over $K$ defined by
$y^2=f_1(x)$ and $y^2 = f_2(x)$ respectively.
Set $F_i(X,Z) = Z^{2g+2}f_i(X/Z) \in K[X,Z]$.
Let $k$ be a field containing $K$.
There exists a $k$-isomorphism from $C_1$ to $C_2$
if and only if there exists $(h,\lambda)\in \operatorname{GL}_2(k)\times k^\times$ such that 
$F_1(h \cdot {}^t(X,Z)) = \lambda^2 F_2(X,Z)$.
%
\end{lem}


\subsection{Cartier-Manin matrix and superspeciality}\label{subsec:HW}

Let $K$ be a perfect field and let $\overline K$ denote the algebraic closure of $K$.
Let $C$ be a nonsingular projective curve over $K$.
We say that $C$ is {\it superspecial}
if its Jacobian $\operatorname{Jac}(C)$ is $\overline K$-isomorphic to the product of
some supersingular elliptic curves.
For a curve $C$ over $K$,
its {\it Cartier-Manin matrix} is defined as a matrix representing the Cartier operator on the space $H^0 (C, \mathrm{\Omega}^1_C)$ of holomorphic differentials of $C$ (cf.\ \cite[Section 2]{Yui}), which 
is uniquely determined as soon as we choose a basis of $H^0 (C, \mathrm{\Omega}^1_C)$.
Here is a well-known method (cf.\ \cite{Gonz}, \cite{Manin-C}, \cite[Section 2]{Yui}) to compute a Cartier-Manin matrix of a hyperelliptic curve.

\begin{prop}\label{prop:HW}
Let $C$ be a hyperelliptic curve $y^2 = f(x)$ of genus $g$ over $K$, where $d = \mathrm{deg} (f)$ is either $2 g +1$ or $2 g + 2$.
Then the $g \times g$ matrix whose $(i,j)$-entry is the coefficient of $x^{p i - j}$ in $f^{(p-1)/2}$ for $1 \leq i , j \leq g$ is a Cartier-Manin matrix of $C$.
\end{prop}


The next corollary follows immediately from the fact that
$C$ is superspecial if and only if
the Cartier operator on the cohomology group $H^0(C,\mathrm{\Omega}^1_C)$ is zero (cf.\ \cite{Nygaard}).

\begin{cor}\label{cor:HW}
Let $C$ be a hyperelliptic curve $y^2 = f(x)$ of genus $g$ over $K$.
Then $C$ is superspecial if and only if the coefficients of $x^{p i - j}$ in $f^{(p-1)/2}$ are equal to $0$ for all integers $i,j$ with $1 \leq i,j \leq g$.
\end{cor}

\subsection{Ingredients to enumerate superspecial hyperelliptic curves}\label{sec:enume}

Assume that $K$ is the finite filed $\mathbb{F}_q$ or its algebraic closure $\overline{\mathbb{F}_q}$, where $q$ is a power of an odd prime $p$.
This subsection reviews a method in \cite{KH18} to enumerate s.sp.\ hyperelliptic curves over $\mathbb{F}_q$.
The same method shall be applied to prove main theorems (Theorems \ref{MainTheorem33} and \ref{MainTheorem44}) in this paper, and it consists of the following three ingredients described precisely in \cite[Section 3]{KH18}: (A) Algorithm to list up superspecial hyperelliptic curves, (B) Reduction of defining equations of hyperelliptic curves, and (C) Isomorphism testing.
Since concrete algorithms for (A) and (C) and a proof of (B) are already given in \cite{KH18}, we here describe only the idea of each ingredient.

\paragraph{(A) Algorithm to list up superspecial hyperelliptic curves:}

In \cite[Section 3.1]{KH18}, we constructed an algorithm with a pseudocode to list up all s.sp.\ hyperelliptic curves of genus $g$ over $\mathbb{F}_q$ for a given $(g,q)$.
The idea is reducing the enumeration of s.sp.\ curves into solving multivariate systems over finite fields (the same idea is also used in a series of papers \cite{KH16}, \cite{KH17-2}, \cite{KH17a}, \cite{KH18tri}).
By Lemma \ref{ReductionSplitNode} below, any hyperelliptic curve of genus $g$ over $\mathbb{F}_q$ is given by the equation $c y^2 = f(x)$ for $c = 1$ or $\epsilon$ with $\epsilon \in  \mathbb{F}_q^\times \smallsetminus (\mathbb{F}_q^\times)^2$, where $f \in \mathbb{F}_q[x]$ is a degree $(2 g + 2)$-polynomial of the form \eqref{eq:reduction} with non-zero discriminant.
For each $c$ and $b$, we derive a multivariate system of algebraic equations from the condition that $c y^2 = f(x)$ is superspecial (i.e., the Cartier-Manin matrix is zero), that is,
\[
\left( \mbox{The coefficient of } x^{p i - j} \mbox{ in } f^{(p-1)/2} \right) = 0
\]
for each $1 \leq i, j \leq g$, where we regard unknown coefficients $a_i$ for $0 \leq i \leq 2 g - 1$ as indeterminates.
For each root of the system, we check whether $f$ has no double root in $\overline{\mathbb{F}_q}$ by constructing the minimal splitting field of $f$.
In this way, we can collect all $f$ of the form \eqref{eq:reduction} with non-zero discriminant such that $c y^2 = f(x)$ is superspecial.

\paragraph{(B) Reduction of defining equations of hyperelliptic curves:}

In \cite[Section 3.2]{KH18}, we gave the following elementary reduction of defining equations of hyperelliptic curves:

\begin{lem}[\cite{KH18}, Lemma 2]\label{ReductionSplitNode}
Assume that $p$ and $2g+2$ are coprime.
Let $\epsilon \in K^\times \smallsetminus (K^\times)^2$.
Any hyperelliptic curve $C$ of genus $g$ over $K$ is the desingularization of
the homogenization of
\begin{eqnarray}
c y^2 = x^{2g+2} + b x^{2g} + a_{2g-1}x^{2g-1} + \cdots + a_1 x + a_0 \label{eq:reduction}
\end{eqnarray}
for $a_i \in K$ for $i=0,1,\ldots, 2g-1$ where $b= 0, 1,\epsilon$ and $c=1,\epsilon$.
\end{lem}

\begin{rem}\label{RemarkQuadraticTwist}
\begin{enumerate}
\item As we pointed out in \cite[Section 3.2]{KH18}, a good method of reduction over an algebraically closed field is to translate three ramified points of the corresponding morphism $C \to \mathbf{P}^1 $ of degree $2$ to $\{0,1,\infty\}$.
However, we can not adopt this method in Lemma \ref{ReductionSplitNode} since the ramified points are not necessarily $K$-rational points.
\item Let $h(x)$ be a monic polynomial over $K$ with non-zero discriminant.
As mentioned in \cite[Remark 3]{KH18}, the hyperelliptic curves $C_1 : y^2 = h (x)$ and $C_2 : \epsilon y^2 = h(x)$ with $\epsilon \in K^\times \smallsetminus (K^\times)^2$ are isomorphic to each other over $K[\sqrt{\epsilon}]$ via $(x,y)\mapsto (x,\sqrt{\epsilon}y)$.
In particular, the superspecialty of $C_1$ is equivalent to that of $C_2$.
\end{enumerate}
\end{rem}

\paragraph{(C) Isomorphism testing:}

We suppose that $p$ and $2 g + 2$ are coprime.
Determining whether two hyperelliptic curves are isomorphic to each other over $K$ is reduced into testing whether a multivariate system has a root over $K$ or not.
Let $C_1$ and $C_2$ be hyperelliptic curves of genus $g$ over $\mathbb{F}_q$.
Recall from Lemma \ref{ReductionSplitNode} that each hyperelliptic curve $C_i$ is the desingularization of the homogenization of $c_i y^2 = f_i (x)$ for $c_i = 1$ or $\epsilon$ with $\epsilon \in \mathbb{F}_q^\times \smallsetminus (\mathbb{F}_q^\times)^2$, where $f_i (x)$ is a polynomial in $\mathbb{F}_q[x]$ of degree $2 g + 2$ with non-zero discriminant.
For each $1 \leq i \leq 2$, let $F_i$ denote the homogenization of $c_i^{-1} f_i$ with respect to an extra variable $z$.
Lemma \ref{lem:isom} shows that $C_1$ and $C_2$ are isomorphic over $K$ if and only if there exist $\lambda \in K^{\times}$ and $h \in \mathrm{GL}_2 (K)$ such that $h \cdot F_1 = \lambda^2 F_2$, where $h \cdot F_1 (x, z) := F_1 ( (x, z) \cdot {}^t h) $.
This is equivalent to that the following multivariate system has a root over $K$:
\begin{eqnarray}
\left\{
\begin{array}{l}
\left( \mbox{All the coefficients in } h \cdot F_1 - \lambda^2 F_2 \right) = 0 \\
\lambda \mu = 1\\
\mathrm{det}(h)  \nu = 1
\end{array}
\right. \label{eq:isom-system}
\end{eqnarray}
where $\lambda$, $\mu$, $\nu$ and all entries of $h$ are indeterminates.
One can decide whether the system \eqref{eq:isom-system} has a root over $K$ or not by computing Gr\"{o}bner bases of the corresponding ideal.
Note that adding field equations such as $\lambda^{q} = \lambda$ is necessary if $K = \mathbb{F}_q$.

\section{Enumeration of superspecial hyperelliptic curves}\label{sec:main}
This section proves Theorems \ref{MainTheorem11} -- \ref{MainTheorem44} stated in Section \ref{sec:intro}.
In Section \ref{subsec:pre-results}, we give a complete proof of computational results in \cite{KH18} for $p \leq 19$ with $g = 4$.
New enumeration results for $(g,q) = (4, 19^2)$ and $(4, 23)$ are stated and proved in Section \ref{subsec:new-results}.
The three ingredients in Section \ref{sec:enume} are applied to computational enumeration for obtaining computational results in Sections \ref{subsec:pre-results} and \ref{subsec:new-results}. 
As a further application, $\mathbb{F}_{p^2}$-maximal curves and $\mathbb{F}_{p^2}$-minimal curves are found in Section \ref{subsec:app} from among enumerated s.sp.\ hyperelliptic curves.

\subsection{Complete proofs of computational results in \cite{KH18}}\label{subsec:pre-results}
After stating results in \cite{KH18} (Propositions \ref{prop:q11and13} -- \ref{prop:q19} below), we prove them by executing enumeration method based on the ingredients in Section \ref{sec:enume}.
Our enumeration method was implemented over Magma V2.22-7~\cite{Magma} in its 64-bit version on a computer with ubuntu 16.04 LTS OS at 3.40 GHz CPU (Intel Core i7-6700) and 15.6 GB memory.
We succeeded in finishing required computation within a day in total.
The source codes and the log files together with detailed information on timing are available at \cite{HPkudo}.

\begin{prop}[\cite{KH18}, Propositions 2 and 3]\label{prop:q11and13}
There does not exist any s.sp.\ hyperelliptic curve of genus $4$ over $\mathbb{F}_q$ defined by an equation of the form \eqref{eq:reduction} for each of $q = 11^2$ and $q=13^2$.
\end{prop}

\begin{prop}[\cite{KH18}, Proposition 4]\label{prop:q17}
There exist exactly five $($resp.\ two$)$ s.sp.\ hyperelliptic curves of genus $4$ over $\mathbb{F}_{17}$, up to isomorphism over $\mathbb{F}_{17}$ $($resp.\ $\overline{\mathbb{F}_{17}})$.
Specifically, the five $\mathbb{F}_{17}$-isomorphisms classes are represented by $C_i : y^2 = f_i (x)$ for $1 \leq i \leq 5$, where
\begin{enumerate}
\item[{\rm (1)}] $f_1 :=  x^{10} + x$,
\item[{\rm (2)}] $f_2 :=  x^{10} + x^7 + 13 x^4 + 12 x $,
\item[{\rm (3)}] $f_3 :=  x^{10} + x^7 + 14 x^6 + 6 x^5 + 12 x^3 + 5 x^2 + 7 x + 6$,
\item[{\rm (4)}] $f_4 :=  x^{10} + x^8 + x^7 + 15 x^6 + 4 x^5 + 12 x^4 + 15 x^3 + 11 x^2 + 9 x + 4$, and
\item[{\rm (5)}] $f_5 :=  x^{10} + x^8 + 2 x^7 + 9 x^5 + x^4 + 10 x^3 + 8 x^2 + 11 x + 16 y^2 + 5$.
\end{enumerate} 
The two $\overline{\mathbb{F}_{17}}$-isomorphism classes are represented by
\begin{enumerate}
\item[{\rm (1)}] $y^2 =  x^{10} + x$, and
\item[{\rm (2)}] $y^2 =  x^{10} + x^7 + 13 x^4 + 12 x $.
\end{enumerate}
\end{prop}

\begin{prop}[\cite{KH18}, Proposition 5]\label{prop:q17-2}
There exist exactly $25$ $($resp.\ two$)$ s.sp.\ hyperelliptic curves of genus $4$ over $\mathbb{F}_{17^2}$, up to isomorphism over $\mathbb{F}_{17^2}$ $($resp.\ $\overline{\mathbb{F}_{17^2}})$.
Specifically, the $25$ $\mathbb{F}_{17^2}$-isomorphisms classes are represented by
\begin{enumerate}
\item[{\rm (1)}] $ y^2=  x^{10} +  x$,
\item[{\rm (2)}] $ y^2=  x^{10} + \zeta x$,
\item[{\rm (3)}] $ y^2=  x^{10} + \zeta^2 x$, 
\item[{\rm (4)}] $ y^2=  x^{10} + \zeta^3 x $,
\item[{\rm (5)}] $ y^2=  x^{10} + \zeta^4 x $,
\item[{\rm (6)}] $ y^2=  x^{10} + \zeta^5 x $,
\item[{\rm (7)}] $ y^2=  x^{10} + \zeta^6 x $,
\item[{\rm (8)}] $ y^2= x^{10} + \zeta^7 x $,
\item[{\rm (9)}] $ y^2= x^{10} + \zeta^8 x $,
\item[{\rm (10)}] $ y^2= x^{10} + x^7 + 13 x^4 + 12 x $,
\item[{\rm (11)}] $ y^2= x^{10} + x^7 + \zeta^{66} x^6 + \zeta^{78} x^5 + \zeta^{138} x^3 + \zeta^{186} x^2 + 7 x + \zeta^{174}$,
\item[{\rm (12)}] $ y^2= x^{10} + \zeta x^7 + \zeta^{74} x^4 + \zeta^{237} x $,
\item[{\rm (13)}] $ y^2= x^{10} + \zeta^2 x^7 + \zeta^{76} x^4 + \zeta^{240} x$, 
\item[{\rm (14)}] $ y^2= x^{10} + x^8 + \zeta^{3} x^7 + \zeta^{98} x^6 + \zeta^{153} x^5 + \zeta^{287} x^4 + \zeta^{71} x^3 + \zeta^{8} x^2 + \zeta^{71} x + \zeta^{254}$,
\item[{\rm (15)}] $ y^2= x^{10} + x^8 + \zeta^3 x^7 + \zeta^{226} x^6 + \zeta^{37} x^4 + \zeta^{147} x^3 + \zeta^{91} x^2 + \zeta^{145} x +  \zeta^{127}$,
\item[{\rm (16)}] $ y^2= x^{10} + x^8 + \zeta^{5} x^7 + \zeta^{138} x^6 + \zeta^{60} x^5 + \zeta^{222} x^4 + \zeta^{128} x^3 + \zeta^{278} x^2 + \zeta^{41} x + \zeta^{24}$,
\item[{\rm (17)}] $ y^2= x^{10} + x^8 + \zeta^{8} x^7 + \zeta^{267} x^6 + \zeta^{50} x^5 + \zeta^{39} x^4 + \zeta^{94} x^3 + \zeta^{58} x^2 + \zeta^{191} x + \zeta^{166}$,
\item[{\rm (18)}] $ y^2= x^{10} + x^8 + \zeta^{10} x^7 + \zeta^{219} x^6 + \zeta^{137} x^5 + \zeta^{103} x^4 + \zeta^{158} x^3 + \zeta^{152} x^2 + \zeta^{2} x  + \zeta^{220}$,
\item[{\rm (19)}] $ y^2= x^{10} + x^8 + \zeta^{13} x^7 + \zeta^{69} x^6 + \zeta^{210} x^5 + \zeta^{22} x^4 + \zeta^{245} x^3 + \zeta^{25} x^2 + \zeta^{35} x + \zeta^{10}$,
\item[{\rm (20)}] $ y^2= x^{10} + x^8 + \zeta^{14} x^7 + \zeta^{104} x^6 + \zeta^{187} x^5 + \zeta^{188} x^4 + 11 x^3 + \zeta^{68} x^2 + \zeta^{148} x  + \zeta^{280}$,
\item[{\rm (21)}] $ y^2= x^{10} + x^8 + \zeta^{16} x^7 + \zeta^{83} x^6 + \zeta^{276} x^5 + \zeta^{164} x^4 + \zeta^{102} x^3 + \zeta^{111} x^2 + \zeta^{2} x  + \zeta^{152}$,
\item[{\rm (22)}] $ y^2= x^{10} + x^8 + \zeta^{16} x^7 + \zeta^{130} x^6 + \zeta^{274} x^5 + \zeta^{133} x^4 + \zeta^{9} x^3 + \zeta^{55} x^2 + \zeta^{175} x  + \zeta^{193}$,
\item[{\rm (23)}] $ y^2= x^{10} + x^8 + \zeta^{19} x^7 + \zeta^{120} x^6 + \zeta^{239} x^5 + \zeta^{123} x^4 + \zeta^{229} x^3 + \zeta^{47} x^2 + \zeta^{145} x + \zeta^{253}$,
\item[{\rm (24)}] $ y^2= x^{10} + x^8 + \zeta^{22} x^7 + \zeta^{250} x^6 + \zeta^{89} x^5 + \zeta^{182} x^4 + \zeta^{9} x^3 + \zeta^{225} x^2 + \zeta^{282} x  + \zeta^{113}$,
\item[{\rm (25)}] $ y^2 = x^{10} + x^8 + \zeta^{41} x^7 + \zeta^{41} x^6 + \zeta^{149} x^5 + \zeta^{169} x^4 + 5 x^3 + \zeta^{197} x^2 + \zeta^{26} x + \zeta^{66}$,
\end{enumerate} 
where we take $\zeta =  - 8 + \sqrt{61} \in \mathbb{F}_{17^2}$, and the two $\overline{\mathbb{F}_{17^2}}$-isomorphism classes are represented by the same equations as those in {\rm Proposition \ref{prop:q17}}.
\end{prop}


\begin{prop}[\cite{KH18}, Proposition 6]\label{prop:q19}
There exist exactly $12$ $($resp.\ two$)$ s.sp.\ hyperelliptic curves of genus $4$ over $\mathbb{F}_{19}$, up to isomorphism over $\mathbb{F}_{19}$ $($resp.\ $\overline{\mathbb{F}_{19}})$.
Specifically, the $12$ $\mathbb{F}_{19}$-isomorphisms classes are represented by
\begin{enumerate}
\item[{\rm (1)}] $y^2 = x^{10} + 1$,
\item[{\rm (2)}] $y^2 = x^{10} + 2$,
\item[{\rm (3)}] $y^2 = x^{10} + x^7 + 4 x^6 + 15 x^5 + 6 x^4 + 8 x^3 + 5 x^2 + 12 x + 1$,
\item[{\rm (4)}] $y^2 = x^{10} + x^8 + 7 x^6 + x^4 + x^2 + 7$,
\item[{\rm (5)}] $y^2 = x^{10} + x^8 + x^7 + 12 x^6 + x^5 + 10 x^4 + 9 x^3 + 8 x^2 + 9 x + 3$,
\item[{\rm (6)}] $y^2= x^{10} + x^8 + x^7 + 13 x^6 + 9 x^5 + 14 x^4 + 4 x^3 + 11 x^2 + 3 x + 8$,
\item[{\rm (7)}] $y^2= x^{10} + x^8 + 2 x^7 + 6 x^6 + 18 x^5 + 4 x^4 + 13 x^3 + 18 x^2 + 10 x + 14$,
\item[{\rm (8)}] $y^2 = x^{10} + x^8 + 2 x^7 + 12 x^6 + 18 x^4 + 5 x^3 + x^2 + 7$,
\item[{\rm (9)}] $y^2= x^{10} + x^8 + 4 x^7 + 8 x^6 + 8 x^5 + 3 x^4 + 11 x^3 + 8 x^2 + 8 x  + 4$,
\item[{\rm (10)}] $y^2 = x^{10} + 2 x^8 + 9 x^6 + 8 x^4 + 16 x^2 + 15$,
\item[{\rm (11)}] $y^2 = x^{10} + 2 x^8 + x^7 + 12 x^6 + 9 x^5 + 2 x^3 + 4 x^2 + 7 x + 4$, and
\item[{\rm (12)}] $y^2 = x^{10} + 2 x^8 + 3 x^7 + 17 x^6 + 9 x^5 + 2 x^3 + 12 x^2 + 2 x + 4$.
\end{enumerate} 
The two $\overline{\mathbb{F}_{19}}$-isomorphism classes are represented by
\begin{enumerate}
\item[{\rm (1)}] $y^2 = x^{10} + 1$, and
\item[{\rm (2)}] $y^2 = x^{10} + x^7 + 4 x^6 + 15 x^5 + 6 x^4 + 8 x^3 + 5 x^2 + 12 x + 1$.
\end{enumerate} 
\end{prop}

\paragraph{{\it Proofs of {\rm Propositions \ref{prop:q11and13} -- \ref{prop:q19}.}}}
Let $g = 4$.
In the procedures 1--4 below, we set the following parameters:
\begin{itemize}
\item ($q = 11^2$) $s_1 := 2 g+1 = 9$ and $s_2:= 2 g -2 = 8$,
\item ($q = 13^2$) $s_1:= 2 g + 1 = 9$ and $s_2:=2 g - 1 = 7$,
\item ($q = 17$ and $19$) $s_1:= 2 g + 1 = 9$ and $s_2 := 2 g - 2 = 6$,
\item ($q=17^2$ and $19^2$) $s_1 := 2 g  = 8$ and $s_2:= 2 g - 2 = 6$,
\item ($q = 23$) $s_1:= 2 g  = 8$ and $s_2 := 2 g - 3 = 5$.
\end{itemize}
By the computer described at the beginning of this section, we conduct the following four procedures:
\begin{enumerate}
\item[{\rm 1}.] Let $a_i$ for $0 \leq i \leq s_{1}-1$ be indeterminates.
\end{enumerate}
For each $(c_{s_1}, \ldots c_{2g-1}) \in (\mathbb{F}_q)^{\oplus 2g-s_1}$ and $c_{2g} \in \{ 0, 1 , \epsilon \}$ with $\epsilon \in \mathbb{F}_q^{\times} \smallsetminus (\mathbb{F}_q^{\times})^2$, proceed with the following procedures:
\begin{enumerate}
\item[{\rm 2}.] Put $f(x):= x^{2 g + 2} + c_{2g} x^{ 2 g} + c_{2 g- 1} x^{2 g - 1} + \cdots +  c_{s_1} x^{s_1} + a_{s_1-1} x^{s_1 - 1} + \cdots + a_1 x + a_0$, and compute $h:= f^{(p-1)/2}$ over $\mathbb{F}_{q} [a_{0}, \ldots, a_{s_1-1}] [x]$.
\item[{\rm 3}.] Let $\mathcal{S} \subset \mathbb{F}_{q} [a_0, \ldots , a_{s_1 -1}] $ be the set of the coefficients of the $g^2$ monomials in $h=f^{(p-1)/2}$, given in Proposition \ref{prop:HW}.
\item[{\rm 4}.] For each $(c_{s_2}, \ldots c_{s_1-1}) \in (\mathbb{F}_q)^{\oplus 2g-s_1-s_2}$, proceed with the following three steps 4a -- 4c:
	\begin{enumerate}
		\item[{\rm 4a}.] Substitute $(c_{s_2}, \ldots c_{s_1-1})$ into $(a_{s_2}, \ldots a_{s_1-1})$ of the coefficients in each $P \in \mathcal{S}$, and put
		\[
		\mathcal{S}^{\prime}:= \{ P (a_0, \ldots , a_{s_2-1}, c_{s_2}, \ldots , c_{s_1-1}) : P \in \mathcal{S} \} \cup \{ a_{i} - c_{i} : s_2 \leq i \leq s_1-1 \}.
		\]
		\item[{\rm 4b}.] With Gr\"{o}bner basis algorithms, compute the roots in $(\mathbb{F}_{q})^{\oplus s_1}$ of the multivariate system $P^{\prime} = 0$ for all $P^{\prime} \in \mathcal{S}^{\prime}$ with variables $a_0, \ldots , a_{s_1 -1}$.
		\item[{\rm 4c}.] For each root $(c_0, \ldots , c_{s_1-1})$ of the system constructed in Step 4b, we set $f_{\rm sol}:= x^{2g+2} + c_{2g} x^{2g} + c_{2g-1} x^{2g-1} + \cdots + c_{s_1} x^{s_1} + c_{s_1-1} x^{s_1-1} + \cdots +  c_1 x + c_0$.
		By constructing the minimal splitting field of $f_{\rm sol}$, test whether $f_{\rm sol}$ has no double root in $\overline{\mathbb{F}_{q}}$ or not. 
		If $f_{\rm sol}$ has no double root in $\overline{\mathbb{F}_{q}}$, store $f_{\rm sol}$.
	\end{enumerate}
\end{enumerate}
As a computational result for each $q$, we obtain the set $\mathcal{F}$ of all the polynomials $f(x)$ of the form in the right hand side of \eqref{eq:reduction} such that $y^2 = f(x)$ are s.sp.\ hyperelliptic curves of genus $g$ over $\mathbb{F}_{q}$.
Put $\mathcal{H}_0 := \{ c y^2 - f(x) : c=1, \epsilon \mbox{ and } f(x) \in \mathcal{F} \}$.
For each pair $(H_1, H_2)$ of elements in $\mathcal{H}_0$ with $H_1 \neq H_2$, the method given in the third paragraph of Section \ref{sec:enume} decides whether $C_1 : H_1 (x,y) = 0$ and $C_2: H_2 (x,y) = 0$ are isomorphic or not.
Finally we obtain the set $\mathcal{H}_{K} \subset \mathcal{H}_0$ of representatives of $K$-isomorphism classes of s.sp.\ hyperelliptic curves of genus $g$ over $\mathbb{F}_q$, where $K$ is either of $\mathbb{F}_q$ and $\overline{\mathbb{F}_q}$.
Propositions \ref{prop:q11and13} -- \ref{prop:q19} follow from the resulting sets $\mathcal{H}_{K}$ with $K = \mathbb{F}_q$ or $\overline{\mathbb{F}_q}$ for $q = 11^2$, $13^2$, $17$, $17^2$ and $19$.\qed

\begin{rem}
In our implementation, the Magma built-in function \textsf{Variety} (resp.\ \textsf{FactorisationOverSplittingField}) was used to solve multivariate systems over finite fields (resp.\ to decide whether a univariate polynomial over a finite field has no double root or not).
\end{rem}

\subsection{New results in characteristic $19$ and $23$}\label{subsec:new-results}

\begin{prop}\label{prop:q19-2}
There exist exactly $18$ $($resp.\ two$)$ s.sp.\ hyperelliptic curves of genus $4$ over $\mathbb{F}_{19^2}$, up to isomorphism over $\mathbb{F}_{19^2}$ $($resp.\ $\overline{\mathbb{F}_{19^2}})$.
Specifically, the $18$ $\mathbb{F}_{19^2}$-isomorphisms classes are represented by
\begin{enumerate}
\item[{\rm (1)}] $y^2= x^{10} + 1$,
\item[{\rm (2)}] $y^2= x^{10} + \zeta$,
\item[{\rm (3)}] $y^2 = x^{10} + \zeta^2$,
\item[{\rm (4)}] $y^2 = x^{10} + \zeta^3$,
\item[{\rm (5)}] $y^2=x^{10} + \zeta^4$,
\item[{\rm (6)}] $y^2 = x^{10} + \zeta^5$,
\item[{\rm (7)}] $y^2 = x^{10} + \zeta^7$,
\item[{\rm (8)}] $y^2 = x^{10} + \zeta^9$,
\item[{\rm (9)}] $y^2 = x^{10} + x^7 + 4 x^6 + 15 x^5 + 6 x^4 + 8 x^3 + 5 x^2 + 12 x + 1$
\item[{\rm (10)}] $y^2 = x^{10} + \zeta^2 x^7 + \zeta^{31} x^6 + \zeta^{169} x^5 + \zeta^{322} x^4 + \zeta^{257} x^3 + \zeta^{352} x^2 + \zeta^{227} x + \zeta^{13}$,
\item[{\rm (11)}] $y^2 = x^{10} + \zeta^2 x^7 + \zeta^{61} x^6 + \zeta^{31} x^5 + \zeta^{286} x^4 + \zeta^{359} x^3 + \zeta^{232} x^2 + \zeta^{245} x + \zeta^7$,
\item[{\rm (12)}] $y^2= x^{10} + x^8 + 2 x^6 + \zeta^{110} x^5 + 7 x^4 + \zeta^{330} x^3 + 9 x^2 + \zeta^{30} x + 17$,
\item[{\rm (13)}] $y^2 = x^{10} + x^8 + x^7 + 13 x^6 + 9 x^5 + 14 x^4 + 4 x^3 + 11 x^2 + 3 x + 8$,
\item[{\rm (14)}] $y^2 = x^{10} + x^8 + \zeta^2 x^7 + \zeta x^6 + \zeta^{54} x^5 + \zeta^{151} x^4 + \zeta^{76} x^3 + \zeta^{205} x^2 + 15 x + \zeta^{335}$,
\item[{\rm (15)}] $y^2 = x^{10} + x^8 + \zeta^2 x^7 + \zeta^{57} x^6 + \zeta^{179} x^5 + x^4 + \zeta^{298} x^3 + \zeta^{89} x^2 + \zeta^{204} x + \zeta^{171}$,
\item[{\rm (16)}] $y^2 = x^{10} + x^8 + \zeta^{12} x^7 + \zeta^{196} x^6 + \zeta^{193} x^5 + \zeta^{281} x^4 + \zeta^{293} x^3 + \zeta^{107} x^2 + \zeta^{316} x + \zeta^{74}$,
\item[{\rm (17)}] $y^2 = x^{10} + x^8 + 2 x^7 + 12 x^6 + 18 x^4 + 5 x^3 + x^2 + 7$,
\item[{\rm (18)}] $y^2 = x^{10} + \zeta x^8 + \zeta^{122} x^6 + \zeta^3 x^4 + \zeta^4 x^2 + \zeta^{125}$,
\end{enumerate}
where we take $\zeta =  - 9 - \sqrt{79} \in \mathbb{F}_{19^2}$, and the two $\overline{\mathbb{F}_{19^2}}$-isomorphism classes are represented by the same equations as those in {\rm Proposition \ref{prop:q19}}.
\end{prop}

\begin{prop}\label{prop:q23}
There exist exactly $14$ $($resp.\ four$)$ s.sp.\ hyperelliptic curves of genus $4$ over $\mathbb{F}_{23}$, up to isomorphism over $\mathbb{F}_{23}$ $($resp.\ $\overline{\mathbb{F}_{23}})$.
Specifically, the $14$ $\mathbb{F}_{23}$-isomorphisms classes are represented by
\begin{enumerate}
\item[{\rm (1)}] $y^2 = x^{10} + x^7 + 3 x^4 + 10 x$,
\item[{\rm (2)}] $y^2 = x^{10} + x^7 + 18 x^4 + 6 x$,
\item[{\rm (3)}] $y^2 = x^{10} + x^7 + 5 x^6 + 3 x^5 + 21 x^4 + 3 x^3 + 9 x^2 + 4 x + 21$,
\item[{\rm (4)}] $y^2 = x^{10} + x^7 + 9 x^6 + 11 x^5 + 19 x^4 + 10 x^3 + 16 x^2 + 8 x + 21$,
\item[{\rm (5)}] $y^2= x^{10} + x^7 + 16 x^6 + 9 x^5 + 14 x^4 + 2 x^3 + 5 x^2 + 6 x + 1$,
\item[{\rm (6)}] $y^2 = x^{10} + x^7 + 17 x^6 + 13 x^5 + 3 x^3 + 14 x^2 + 20 x + 15$,
\item[{\rm (7)}] $y^2= x^{10} + x^7 + 18 x^6 + 21 x^4 + x^3 + 8 x^2 + 20 x + 21$,
\item[{\rm (8)}] $y^2 = x^{10} + x^8 + 3 x^6 + 2 x^4 + 2 x^2 + 6$,
\item[{\rm (9)}] $y^2= x^{10} + x^8 + 6 x^6 + 22 x^4 + 4 x^2 + 3$,
\item[{\rm (10)}] $y^2 = x^{10} + x^8 + 8 x^6 + 7 x^4 + 15 x^2 + 14$,
\item[{\rm (11)}] $y^2= x^{10} + x^8 + 2 x^7 + 6 x^6 + 3 x^5 + 14 x^4 + 16 x^3 + 11 x^2 + 19$,
\item[{\rm (12)}] $y^2 = x^{10} + x^8 + 3 x^7 + x^6 + 13 x^4 + 22 x^3 + 12 x^2 + 4$,
\item[{\rm (13)}] $y^2= x^{10} + x^8 + 4 x^7 + 12 x^6 + 2 x^5 + 12 x^2 + 11 x + 18$,
\item[{\rm (14)}] $y^2 = x^{10} + x^8 + 5 x^7 + 15 x^6 + 22 x^5 + 11 x^4 + 7 x^2 + 18 x + 17$.
\end{enumerate} 
The four $\overline{\mathbb{F}_{23}}$-isomorphism classes are represented by
\begin{enumerate}
\item[{\rm (1)}] $y^2 = x^{10} + x^7 + 3 x^4 + 10 x$,
\item[{\rm (2)}] $y^2 = x^{10} + x^7 + 18 x^4 + 6 x$,
\item[{\rm (3)}] $y^2 = x^{10} + x^7 + 5 x^6 + 3 x^5 + 21 x^4 + 3 x^3 + 9 x^2 + 4 x + 21$, and
\item[{\rm (4)}] $y^2 = x^{10} + x^7 + 9 x^6 + 11 x^5 + 19 x^4 + 10 x^3 + 16 x^2 + 8 x + 21$.
\end{enumerate} 
\end{prop}

\subsection{Application to finding maximal curves and minimal curves}\label{subsec:app}
Since any maximal or minimal (hyperelliptic) curve over $\mathbb{F}_{p^2}$ is superspecial, it is included in enumerated s.sp.\ curves over $\mathbb{F}_{p^2}$ if exists.
We have non-existence results (Corollary \ref{cor:max121and169} below) from Theorem \ref{MainTheorem11} for $p= 11$ and $13$, whereas we found $\mathbb{F}_{p^2}$-maximal curves and $\mathbb{F}_{p^2}$-minimal curves for $p=17$, $19$ and $23$ (Corollaries \ref{cor:max289} -- \ref{cor:max23} below). 
Using a computer, we found them by computing the number of $\mathbb{F}_{p^2}$-rational points on s.sp.\ curves in Propositions \ref{prop:q17-2} -- \ref{prop:q23}.
See also a table at \cite{HPkudo} for explicit $\mathbb{F}_{p^2}$-maximal or $\mathbb{F}_{p^2}$-minimal curves defined over $\mathbb{F}_{p}$ with $p=17$ and $19$, which we omit write down here.


\begin{cor}[\cite{KH18}, Corollaries 3 and 4]\label{cor:max121and169}
There does not exist any $\mathbb{F}_{p^2}$-maximal $($resp.\ minimal$)$ hyperelliptic curve of genus $4$ for each of $p=11$ and $13$.
\end{cor}

\begin{cor}[\cite{KH18}, Corollary 5]\label{cor:max289}
There exists exactly two $($resp.\ two$)$ $\mathbb{F}_{17^2}$-maximal $($resp.\ $\mathbb{F}_{17^2}$-minimal$)$ hyperelliptic curves of genus $4$ over $\mathbb{F}_{17^2}$ up to isomorphism over $\mathbb{F}_{17^2}$.
Specifically, the two maximal curves are given by
\begin{eqnarray}
y^2 &=&  x^{10} +  x, \nonumber \\
y^2 &=& x^{10} + x^7 + 13 x^4 + 12 x , \nonumber
\end{eqnarray}
respectively.
The two minimal curves are given by
\begin{eqnarray}
y^2 & = & x^{10} + x^8 + \zeta^{16} x^7 + \zeta^{83} x^6 + \zeta^{276} x^5 + \zeta^{164} x^4 + \zeta^{102} x^3 + \zeta^{111} x^2 + \zeta^{2} x  + \zeta^{152}, \nonumber \\
y^2 & = & x^{10} + x^8 + \zeta^{22} x^7 + \zeta^{250} x^6 + \zeta^{89} x^5 + \zeta^{182} x^4 + \zeta^{9} x^3 + \zeta^{225} x^2 + \zeta^{282} x  + \zeta^{113} \nonumber 
\end{eqnarray}
respectively, where we take $\zeta =  - 8 + \sqrt{61} \in \mathbb{F}_{17^2}$.
\end{cor}

\begin{cor}[\cite{KH18}, Corollary 6]\label{cor:max361}
There exists exactly two $($resp.\ two$)$ $\mathbb{F}_{19^2}$-maximal $($resp.\ $\mathbb{F}_{19^2}$-minimal$)$ hyperelliptic curves of genus $4$ over $\mathbb{F}_{19^2}$ up to isomorphism over $\mathbb{F}_{19^2}$.
Specifically, the two maximal curves are given by
\begin{eqnarray}
y^2 &=&  x^{10} +  1, \nonumber \\
y^2 &=&  x^{10} + x^7 + 4 x^6 + 15 x^5 + 6 x^4 + 8 x^3 + 5 x^2 + 12 x + 1, \nonumber
\end{eqnarray}
respectively.
The two minimal curves are given by
\begin{eqnarray}
y^2 & = & x^{10} + x^8 + \zeta^2 x^7 + \zeta^{57} x^6 + \zeta^{179} x^5 + x^4 + \zeta^{298} x^3 + \zeta^{89} x^2 + \zeta^{204} x + \zeta^{171}, \nonumber  \\
y^2 &=& x^{10} + x^8 + 2 x^7 + 12 x^6 + 18 x^4 + 5 x^3 + x^2 + 7, \nonumber 
\end{eqnarray}
respectively, where we take $\zeta =  - 9 - \sqrt{79} \in \mathbb{F}_{19^2}$.
\end{cor}

\begin{cor}\label{cor:max23}
There exist $\mathbb{F}_{23^2}$-maximal hyperelliptic curves of genus $4$ defined over $\mathbb{F}_{23}$.
There also exists an $\mathbb{F}_{23^2}$-minimal hyperelliptic curve of genus $4$ over $\mathbb{F}_{23}$.
Specifically, the following $11$ hyperelliptic curves over $\mathbb{F}_{23}$ are $\mathbb{F}_{23^2}$-maximal:
\begin{eqnarray}
y^2 & = & x^{10} + x^7 + 3 x^4 + 10 x, \nonumber \\
y^2 & = & x^{10} + x^7 + 18 x^4 + 6 x, \nonumber \\
y^2 & = & x^{10} + x^7 + 5 x^6 + 3 x^5 + 21 x^4 + 3 x^3 + 9 x^2 + 4 x + 21, \nonumber \\
y^2 & = & x^{10} + x^7 + 9 x^6 + 11 x^5 + 19 x^4 + 10 x^3 + 16 x^2 + 8 x + 21, \nonumber \\
y^2 & = & x^{10} + x^7 + 16 x^6 + 9 x^5 + 14 x^4 + 2 x^3 + 5 x^2 + 6 x + 1, \nonumber \\
y^2 & = & x^{10} + x^7 + 18 x^6 + 21 x^4 + x^3 + 8 x^2 + 20 x + 21, \nonumber \\
y^2 & = & x^{10} + x^8 + 3 x^6 + 2 x^4 + 2 x^2 + 6, \nonumber \\
y^2 & = & x^{10} + x^8 + 6 x^6 + 22 x^4 + 4 x^2 + 3, \nonumber \\
y^2 & = & x^{10} + x^8 + 8 x^6 + 7 x^4 + 15 x^2 + 14, \nonumber \\
y^2 & = &  x^{10} + x^8 + 4 x^7 + 12 x^6 + 2 x^5 + 12 x^2 + 11 x + 18, \nonumber \\
y^2 & = & x^{10} + x^8 + 5 x^7 + 15 x^6 + 22 x^5 + 11 x^4 + 7 x^2 + 18 x + 17. \nonumber
\end{eqnarray}
On the other hand, the following curve over $\mathbb{F}_{23}$ is $\mathbb{F}_{23^2}$-minimal:
\begin{eqnarray}
y^2 & = & x^{10} + x^8 + 2 x^7 + 6 x^6 + 3 x^5 + 14 x^4 + 16 x^3 + 11 x^2 + 19. \nonumber
\end{eqnarray}
\end{cor}

\begin{rem}
The maximal hyperelliptic curve $y^2 = x^{10} + x$ (resp.\ $y^2 = x^{10} + 1$) over $\mathbb{F}_{17^2}$ (resp.\ $\mathbb{F}_{19^2}$) is of known type, see e.g., \cite{Taf} for more general results on the existence of such a kind of maximal hyperelliptic curves.
\end{rem}

\section{Computing automorphism groups of enumerated hyperelliptic curves}\label{sec:aut}
In this section, we present an algorithm to compute the automorphism group of a hyperelliptic curve over $K= \mathbb{F}_q$ or $\overline{\mathbb{F}_q}$, where $q$ is a power of a prime $p>2$.
Note that our algorithm works for not only superspecial but also arbitrary hyperelliptic one such that $p$ and $2 g + 2$ are coprime.
As an application of the algorithm, this section also studies the automorphism groups of s.sp.\ hyperelliptic curves of genus $4$ enumerated in Section \ref{sec:main}.
Moreover, we check that our enumeration in Section \ref{sec:main} is compatible with the Galois cohomology theory.

\subsection{Description of automorphism groups of hyperelliptic curves}
Assume that $p$ and $2 g + 2$ are coprime.
Let $C$ be a hyperelliptic curve of genus $g$ over $K$ defined by $y^2 = f (x)$ for some monic polynomial $f (x)$ in $K[x]$ of degree $2 g+ 2$.
Let $F$ denote the homogenization of $f$ with respect to an extra variable $z$.

In order to describe the automorphism group of $C$, we consider the groups
\begin{eqnarray*}
\tilde{G}_K &=& \{ (h,\lambda)\in \GL_2(K)\times K^\times \mid F(h\cdot{}^t(X,Z)) = \lambda^2 F(X,Z) \},\\
G_K &=& \{ h\in \GL_2(K) \mid F(h\cdot{}^t(X,Z)) = F(X,Z) \}
\end{eqnarray*}
and $\mu\!\!\!\mu_{g+1}(K) = \{a\in K^\times \mid a^{g+1} = 1\}$,
where $G_K$ is considered as a subgroup of $\tilde{G}_K$, via the homomorphism
sending $h\in G_K$ to $(h,1)\in \tilde{G}_K$. 

\begin{lem}
There exists a diagram
\[\begin{CD}
1 @>>> K^\times @>\psi>> \tilde{G}_K @>\varphi>> \Aut_K(C) @>>>1\\
@. @AA\cup A @AA\cup A @| @.\\
1 @>>> \mu\!\!\!\mu_{g+1}(K) @>>> G_K @>\varphi|_{G_K}>> \Aut_K(C) @.
\end{CD}\]
with exact horizontal sequences,
where $\varphi$ sends $\left(\begin{pmatrix}\alpha & \beta\\
\gamma & \delta\end{pmatrix},\lambda\right)$ to the automorphism
maping $(x,y)$ to $\left(\frac{\alpha x+\beta}{\gamma x+\delta},\frac{\lambda y}{(\gamma x+\delta)^{g+1}}\right)$ and $\psi$ sends $u$ to $\left(\begin{pmatrix}u & 0 \\ 0 & u\end{pmatrix}, u^{g+1}\right)$.
Moreover, if the $(g+1)$-th power map on $K^\times$ is surjective,
then, $\varphi|_{G_K}$ is surjective.
\end{lem}
\begin{proof}
The surjectivity of $\varphi$ holds, as any automorphism is given in the form as in the lemma (cf. the proof of \cite[Lemma 1]{KH18}).
Then obviously the kernel of $\varphi$ is equal to the image of $\psi$.
The lower exact sequence is immediately obtained by restricting the upper one to $G_K$. If the $(g+1)$-th power map on $K^\times$ is surjective, then
any automorphism determined by $\left(\begin{pmatrix}\alpha & \beta\\
\gamma & \delta\end{pmatrix},\lambda\right) \in {\tilde G}_K$
is equal to that determined by $\left(\mu^{-1}\begin{pmatrix}\alpha & \beta\\
\gamma & \delta\end{pmatrix},1\right)$ for $\mu$ with $\mu^{g+1}=\lambda$,
whence $\varphi|_{G_K}$ is surjective.
\end{proof}

Consider the case of $g=4$.
Applying this lemma for an algebraically closed field $k$, we have
$\Aut(C) \simeq G_k/\mu\!\!\!\mu_5(k)$.
If $K=\F_p$ for $p=17,19,23$, 
the group $\mu\!\!\!\mu_5(K)$ is trivial, whence
we have $\Aut_K(C) \simeq G_K$.

%

\subsection{Algorithm to compute automorphism groups}\label{sec:algAut}
In this subsection, we present an algorithm to compute the automorphism group of a hyperelliptic curve over $\mathbb{F}_q$ or $\overline{\mathbb{F}_q}$.
Note that our algorithm works for not only superspecial but also arbitrary hyperelliptic one such that $p$ and $2 g + 2$ are coprime.
Let $C$ be a hyperelliptic curve of genus $g$ over $K$ defined by $c y^2 = f (x)$ for some polynomial $f (x)$ in $K[x]$ of degree $2 g+ 2$.
Let $F$ denote the homogenization of $f$ with respect to an extra variable $Z$.

\paragraph{Algorithm to compute $\mathrm{Aut}_K (C)$.}
For the input $(c, f(x) )$, and $q$ as above, the following 5 steps compute $G_K$ for $K = \mathbb{F}_q $, or $K = \overline{\mathbb{F}_q}$:
\begin{enumerate}
	\item Let $b_1$, $b_2$, $b_3$, $b_4$ and $\nu$ be indeterminates, and set
		\[
		F (X, Z) := c^{-1} Z^{2 g + 2 }f(X/Z) \quad \mbox{and} \quad
		h:= 
		\begin{pmatrix}
		b_1 & b_2\\
		b_3 & b_4
		\end{pmatrix},
		\]
		where $h$ is a square matrix whose entries are indeterminates.
	\item Compute $F'(X, Z):=F ( (X, Z) \cdot {}^t h) - F (X, Z)$ over the polynomial ring $K [b_1, b_2, b_3, b_4, \nu] [X,Z]$ whose coefficient ring is also a polynomial ring.
	\item Let $\mathcal{C}_F$ be the set of the coefficients of the non-zero terms in $F'(X,Z)$.
	We put
	\begin{eqnarray}
	\mathcal{C} := \mathcal{C}_F \cup \{ \mathrm{det} (h) \nu - 1 \}. \nonumber
	\end{eqnarray}
	For $K = \mathbb{F}_q$, we replace $\mathcal{C}$ by
	\[
	\mathcal{C} \cup \{ b_i^q - b_i : 1 \leq i \leq 4 \} \cup \{ \nu^q - \nu \}.
	\]
	\item Compute $V_K (\langle \mathcal{C} \rangle)$ in $\mathbb{A}^5(K)$, where $V_K (\langle \mathcal{C} \rangle)$ denotes the set of zeros of the ideal $\langle \mathcal{C} \rangle \subset K [b_1, b_2, b_3, b_4, \nu]$.
	One can do this by computing Gr\"{o}bner bases.
	\item For the computed finite group $G_K$ given by
	\[
	G_K = \left\{ \begin{pmatrix}
		b_1 & b_2\\
		b_3 & b_4
		\end{pmatrix} : (b_1, b_2, b_3, b_4, \nu) \in V_K (\langle \mathcal{C} \rangle) \mbox{ for some } \nu \in K^{\times}
		\right\} ,
	\]
	compute $G_K / \mu\!\!\!\mu_5(K)$, and then output it.
\end{enumerate}

\subsection{Automorphism groups of enumerated superspecial hyperelliptic curves}

Executing the algorithm in Subsection \ref{sec:algAut} over the computer algebra system Magma, we compute automorphism groups of s.sp.\ hyperelliptic curves of genus $4$ enumerated in Section \ref{sec:main}.
All elements of $G_K$ for each automorphism group $\mathrm{Aut}_K (C)$ is computed, and $\mathrm{Aut}_K (C)$ is also computed as the set of representatives of a quotient group of $G_{K}$.
The orders of automorphism group over $\mathbb{F}_p$ with those over $\overline{\mathbb{F}_{p}}$ are summarized in Table \ref{tab:ACAut}.
Note that the group structure of each automorphism group is determined by Magma's built-in function \texttt{GroupName}.
For each integer $t > 0$, we denote by ${\rm S}_t$ and ${\rm D}_t$ the symmetric group of degree $t$ and the dihedral group of degree $t$, respectively.



\renewcommand{\arraystretch}{1.3}
\begin{table}[H]
\begin{center}
\begin{tabular}{c|ccc|ccc} \hline
$p$ & Superspecial hyperelliptic & \multirow{2}{*}{$\mathrm{Aut}(C)$} & \multirow{2}{*}{$\# \mathrm{Aut}(C)$} & $\mathbb{F}_{p}$-forms & \multirow{2}{*}{$\mathrm{Aut}_{\mathbb{F}_{p}}(C^{\prime})$} & \multirow{2}{*}{$\# \mathrm{Aut}_{\mathbb{F}_{p}}(C^{\prime})$} \\ 
& curves $C$ over $\overline{\mathbb{F}_{p}}$ & & & $C^{\prime}$ of $C$ & & \\ \hline
\multirow{5}{*}{$17$} & \multirow{2}{*}{$C_1^{\rm (alc)}$} & \multirow{2}{*}{$\mathrm{C}_{18}$} & \multirow{2}{*}{$18$} & $C_1$ & ${\rm C}_2$ & $2$ \\
& & & & $C_4$ & ${\rm C}_2$ & $2$ \\ \cline{2-7}
& \multirow{3}{*}{$C_2^{\rm (alc)}$} & \multirow{3}{*}{$\mathrm{SL}_2 (\mathbb{F}_3)$} & \multirow{3}{*}{$24$} & $C_2$ & ${\rm C}_2$  & $2$ \\
 & & & & $C_3$ & ${\rm C}_4$ & $4$ \\
& & & & $C_5$ & ${\rm C}_4$ & $4$ \\ \hline
\multirow{12}{*}{$19$} & \multirow{9}{*}{$C_1^{\rm (alc)}$} & \multirow{9}{*}{${\rm C}_5 \rtimes {\rm D}_4$} & \multirow{9}{*}{$40$} & $C_1$ & ${\rm D}_4$ & $8$ \\
 & & & & $C_2$ & ${\rm C}_2 \times {\rm C}_2$ & $4$ \\
  & & & & $C_4$ & ${\rm D}_4$ & $8$ \\
   & & & & $C_6$ & ${\rm C}_{10}$ & $10$ \\
    & & & & $C_8$ & ${\rm C}_5 \rtimes {\rm C}_4$ & $20$ \\
     & & & & $C_9$ & ${\rm C}_{10}$ & $10$ \\
      & & & & $C_{10}$ & ${\rm D}_{10}$ & $20$ \\
       & & & & $C_{11}$ & ${\rm C}_{10}$ & $10$ \\
& & & & $C_{12}$ & ${\rm C}_{10}$ & $10$ \\ \cline{2-7}
& \multirow{3}{*}{$C_2^{\rm (alc)}$} & \multirow{3}{*}{${\rm D}_4$} & \multirow{3}{*}{$8$} & $C_3$ & ${\rm S}_2$  & $2$ \\
 & & & & $C_5$ & ${\rm C}_4$ & $4$ \\
& & & & $C_7$ & ${\rm C}_4$ & $4$ \\ \hline
\multirow{14}{*}{$23$} & \multirow{2}{*}{$C_1^{\rm (alc)}$} & \multirow{2}{*}{$\mathrm{C}_6$} & \multirow{2}{*}{$6$} & $C_1$ & ${\rm C}_2$ & $2$ \\
 & & & & $C_5$ & ${\rm C}_2$ & $2$ \\ \cline{2-7}
 & \multirow{3}{*}{$C_2^{\rm (alc)}$} & \multirow{3}{*}{$\mathrm{SL}_2 (\mathbb{F}_3)$} & \multirow{3}{*}{$24$} & $C_2$ & ${\rm C}_2$  & $2$ \\
  & & & & $C_6$ & ${\rm C}_4$ & $4$ \\
   & & & & $C_{12}$ & ${\rm C}_4$ & $4$ \\ \cline{2-7}
   & \multirow{5}{*}{$C_3^{\rm (alc)}$} & \multirow{5}{*}{$\mathrm{D}_4$} & \multirow{5}{*}{$8$} & $C_3$ & ${\rm C}_2 \times {\rm C}_2$  & $4$ \\
    & & & & $C_8$ & ${\rm D}_4$ & $8$ \\
     & & & & $C_{10}$ & ${\rm C}_2 \times {\rm C}_2$ & $4$ \\
     & & & & $C_{11}$ & ${\rm C}_4$ & $4$ \\
      & & & & $C_{14}$ & ${\rm D}_4$ & $8$ \\ \cline{2-7}
& \multirow{4}{*}{$C_4^{\rm (alc)}$} & \multirow{4}{*}{${\rm C}_2 \times {\rm C}_2$} & \multirow{4}{*}{$4$} & $C_4$ & ${\rm C}_2 \times {\rm C}_2$  & $4$ \\
 & & & & $C_7$ & ${\rm C}_2 \times {\rm C}_2$ & $4$ \\
 & & & & $C_9$ & ${\rm C}_2 \times {\rm C}_2$ & $4$ \\
& & & & $C_{13}$ & ${\rm C}_2 \times {\rm C}_2$ & $4$ \\ \hline
\end{tabular}
\caption{The automorphism group $\mathrm{Aut} (C^{\rm (alc)}):=\mathrm{Aut}_{\overline{\mathbb{F}_{p}}} (C^{\rm (alc)})$ of the superspecial hyperelliptic curve $C^{\rm (alc)}$ over $\overline{\mathbb{F}_{p}}$ and the automorphism groups $\mathrm{Aut}_{\mathbb{F}_{p}} (C_i)$ of the superspecial hyperelliptic curves 
$C_i$ over $\mathbb{F}_{p}$ for $17 \leq p \leq 23$.
}\label{tab:ACAut}
\end{center}
\end{table}

\subsection{Compatibility with the Galois cohomology theory}
This subsection shows that our enumeration in Section \ref{sec:main} is compatible with the Galois cohomology theory.
Specifically, we check the below equalities \eqref{eq:Galois1} -- \eqref{eq:Galois3} deduced from the Galois cohomology theory for each of enumerated s.sp.\ hyperelliptic curves over the algebraic closure.

For a hyperelliptic curve $C$ over $\overline{\mathbb{F}_q}$, two elements $a$ and $b$ in $\mathrm{Aut}(C)$ are said to be $\sigma$-{\it conjugate} if $a = g^{-1} b g^{\sigma}$ for some $g \in \mathrm{Aut}(C)$, where $\sigma$ is the Frobenius on $\mathrm{Aut} (C)$.
For an element $a \in \mathrm{Aut}(C)$, we denote by ${\operatorname{Orb}}(a)$ the orbit of $a$, i.e., ${\operatorname{Orb}}(a) :=\{ g^{-1} a g^{\sigma} : g \in \mathrm{Aut}(C) \}$, called the $\sigma$-conjugacy class of $a$.
The $\sigma$-stabilizer of $a$ is defined as the subgroup $\{ g \in \mathrm{Aut}(C) : a = g^{-1} a g^{\sigma} \}$, written $\mathrm{Aut} (C)_a$.

By the Galois cohomology theory, we have the following well-known facts:
\begin{itemize}
\item For a hyperelliptic curve $C$ over $\overline{\mathbb{F}_q}$, we have
\begin{eqnarray}
|\mathrm{Aut}(C ) / \sigma\mbox{\rm -conjugacy}| = (\mbox{The number of $\mathbb{F}_{q}$-forms of $C$}), \label{eq:Galois1}
\end{eqnarray}
and thus 
\begin{eqnarray}
\sum_{C \in \mathrm{SSp}\text{-}\mathrm{Hyp}_g (\overline{\mathbb{F}_q})} |\mathrm{Aut}(C ) / \sigma\mbox{\rm -conjugacy}| = | \mathrm{SSp}\text{-}\mathrm{Hyp}_g (\mathbb{F}_q) |, \label{eq:Galois2}
\end{eqnarray}
where $\mathrm{SSp}\text{-}\mathrm{Hyp}_g (K)$ denotes the set of $K$-isomorphism classes of s.sp.\ hyperelliptic curves over $K$, where $K$ is a finite field or its algebraic closure.
\item For a hyperelliptic curve $C$ over $\overline{\mathbb{F}_q}$ and for each element $a \in \mathrm{Aut} (C)$, there exists a bijection $\mathrm{Aut} (C)_a \simeq \mathrm{Aut}_{\mathbb{F}_q}(C^{(a)} )$.
Here $C^{(a)}$ denotes the $\mathbb{F}_q$-form associated to $a$ via the isomorphism
\[
H^1 (\mathrm{Gal} (\overline{\mathbb{F}_q}/\mathbb{F}_q), \mathrm{Aut}(C)) \cong \mathrm{Aut}(C ) / \sigma\mbox{\rm -conjugacy} .
\]
From the orbit-stabilizer theorem, we have
$| \mathrm{Aut} (C) | / |{\operatorname{Orb}}(a)| = | \mathrm{Aut} (C)_a |$,
and thus
\begin{eqnarray}
 | \mathrm{Aut} (C) | / |{\operatorname{Orb}}(a)| = | \mathrm{Aut}_{\mathbb{F}_q}(C^{(a)} ) |. \label{eq:Galois3}
\end{eqnarray}
\end{itemize}
For each $\mathrm{Aut} (C)$ in Table \ref{tab:ACAut}, we determine the left hand side of each of \eqref{eq:Galois1} -- \eqref{eq:Galois3} with $(g,q) = (4,17)$, $(4,19)$ or $(4,23)$ by computing $\sigma$-conjugacy classes of $\mathrm{Aut} (C)$ over Magma.
As a result, we confirmed that the equalities \eqref{eq:Galois1} -- \eqref{eq:Galois3} hold, where the value of the right hand side of \eqref{eq:Galois1} (resp.\ \eqref{eq:Galois3}) is already obtained in computation to prove Propositions \ref{prop:q17}--\ref{prop:q23} (resp.\ results in Table \ref{tab:ACAut}). 
For details of computational results, see \cite{HPkudo}, where codes over Magma and log files are available.

\subsection*{Acknowledgments}
This work was supported by JSPS Grant-in-Aid for Research Activity Start-up 18H05836 and , and JSPS Grant-in-Aid for Scientific Research (C) 17K05196.

\footnote[0]{
E-mail address of the first author: \texttt{m-kudo@math.kyushu-u.ac.jp}\\
E-mail address of the second author: \texttt{harasita@ynu.ac.jp}
}

\end{document}